\documentclass[reqno, oneside]{amsart}

\usepackage{amssymb}
\usepackage{enumerate}
\usepackage[all]{xy}
\usepackage{ifpdf}

\ifpdf
  \usepackage[pdftex]{graphicx}
  \usepackage[colorlinks=true, pdftex]{hyperref}
\else
  \usepackage[dvipdfm]{graphicx}
  \usepackage[colorlinks=true, dvipdfm]{hyperref}
\fi

\usepackage{subfig}
\usepackage{xspace}

\hypersetup{%
pdftitle={The minimal degree of plane models of double covers of smooth curves},%
pdfauthor={Department of Mathematics, Chungnam National University, Daejeon 305-764, Korea},%
pdfkeywords={double covers, simple nets, primitive linear series},%
citecolor=black,%
filecolor=black,%
linkcolor=black,%
urlcolor=black,%
}

\newtheorem{theorem}{Theorem}[section]

\newtheorem{lemma}[theorem]{Lemma}
\newtheorem{proposition}[theorem]{Proposition}

\newtheorem*{maintheorem}{Main Theorem}

\theoremstyle{definition}

\theoremstyle{remark}
\newtheorem{remark}[theorem]{Remark}

\numberwithin{equation}{section} %

\def\sheaf#1{\ensuremath \mathcal#1}

\newcommand{\elm}{\ensuremath elm}%

\DeclareMathOperator{\supp}{Supp}%
\DeclareMathOperator{\divisorgroup}{Div}%
\DeclareMathOperator{\gon}{gon}%

\newcommand{\imageofblup}[1]{\ensuremath \widetilde{#1}}

\newcommand{\linsys}[1]{\ensuremath \lvert #1 \rvert}
\newcommand{\blup}{\ensuremath \varphi}
\newcommand{\bldown}{\ensuremath \psi}

\newcommand{\PP}{\ensuremath \mathbf{P}}
\newcommand{\lineqv}{\ensuremath \sim}

\newcommand{\deFranchis}{\ensuremath \mathbb{H}}
\newcommand{\CC}{\ensuremath \mathbb{C}}

\newcommand{\claiminline}{\textit{Claim}.~}

\begin{document}
\title[The minimal degree of plane models of double covers]{The minimal degree of plane models of double covers of smooth curves}

\author{Dongsoo Shin}

\address{Department of Mathematics, Chungnam National University, Daejeon 305-764, Korea}

\email{dsshin@cnu.ac.kr}



\begin{abstract}
If $X$ is a smooth curve such that the minimal degree of its plane models is not too small compared with its genus, then $X$ has been known to be a double cover of another smooth curve $Y$ under some mild condition on the genera. However there are no results yet for the minimal degrees of plane models of double covers except some special cases. In this paper, we give upper and lower bounds for the minimal degree of plane models of the double cover $X$ in terms of the gonality of the base curve $Y$ and the genera of $X$ and $Y$. In particular, the upper bound equals to the lower bound in case $Y$ is hyperelliptic. We give an example of a double cover which has plane models of degree equal to the lower bound.
\end{abstract}

\maketitle


\section{Introduction}

we work over the field $\CC$ of complex numbers and a curve is an irreducible complete algebraic curve over $\CC$. Let $X$ be a smooth curve of genus $g_x > 0$. There is a non-special very ample divisor of degree $d+1$ on $X$ for any $d \ge g_x + 2$ by Halphen's theorem, which induces an embedding $\phi : X \hookrightarrow \PP^{r}$, where $r = d+1-g_x$. A general projection $\pi : \PP^r - \PP^{r-3} \to \PP^3$ maps  $\phi(X)$ biregularly to a non-degenerate smooth curve $X'$ of degree $d+1$. Then a general projection $\pi : \PP^3 - \{P\} \to \PP^2$ with center $P \in X'$ maps $X'$ birationally onto a plane curve of degree $d$. Therefore $X$ has a plane model of degree $d$ for any $d \ge g_x + 2$, which shows that there is no maximum among degrees of plane models of $X$.

The minimal degree of plane models of $X$, which is denoted by $s_X(2)$, is what we are interested in. By the above discussion, $s_X(2) \le g_x + 2$. It has been classically known that $s_X(2) = [\frac{2(g_x + 4)}{3}]$ for a general curve $X$; cf.~Severi~\cite{Severi-1968}. On the other hand, if $s_{X}(2)$ is not too small compared with the genus $g_x$, then $X$ is a double cover of another smooth curve: Speaking more precisely, if $s_{X}(2) = g_x + 2 - t$ for some $t \ge 0$, then $X$ is a double cover of a smooth curve of genus at most $t$ provided that $g_x$ is sufficiently large with respect to $t$;  Keem-Martens~\cite[3.2]{Keem-Martens-2006}. For instance, if $s_{X}(2) = g_x + 2$, then $X$ is hyperelliptic; cf.~Coppens-Martens~\cite[2.2]{Coppens-Martens-2000}. If $s_{X}(2) = g_x + 1$, then $X$ is bi-elliptic (i.e., a double cover of an elliptic curve) provided that $g_x \ge 6$; Coppens-Martens~\cite[2.5]{Coppens-Martens-2000}. If $s_X(2) = g_x$ and $g_x \ge 16$, then $X$ is a double cover of a curve of genus $2$; Keem-Martens~\cite[4.1]{Keem-Martens-2006}.

It would be an interesting problem to compute $s_X(2)$ for double covers. Let $X$ be a double cover of a smooth curve $Y$ of genus $g_y$. There are few results in this direction: For example, if $X$ is hyperelliptic ($g_y=0$), then $s_X(2) = g_x+2$, and, if $X$ is bi-elliptic ($g_y=1$), then $s_X(2) = g_x + 1$ if $g_x \ge 4$; Coppens-Martens~\cite[2.2]{Coppens-Martens-2000}. To the author's knowledge, there is no further result yet on computing $s_X(2)$ in case $g_y \ge 2$.

In this paper, we give a bound for $s_X(2)$ in case $X$ is a double cover of a smooth curve $Y$ of genus $g_y \ge 2$.

\begin{maintheorem}
Let $X$ be a smooth curve of genus $g_x$ which is a double cover of a smooth curve $Y$ of genus $g_y \ge 2$. If $g_x \ge 4g_y-2$, then
\begin{equation*}
g_x - 2g_y + 1 + \gon(Y) \le s_X(2),
\end{equation*}
where $\gon(Y)$ is the gonality of $Y$. If $g_x \ge 8g_y + 2$, then
\begin{equation*}
s_X(2) \le g_x - 2g_y - 1 + 2\gon(Y).
\end{equation*}
\end{maintheorem}

Furthermore, for any given smooth curve $Y$ of genus $g_y \ge 2$, we construct a double cover $X$ of $Y$ with genus $g_x$ attaining the minimal possible degree $g_x - 2g_y + 1 + \mathrm{gon}(Y)$ of plane models under some mild condition on the genera, which shows that the lower bound $g_x - 2g_y + 1 + \mathrm{gon}(Y)$ is sharp.

The main idea of the proof is to investigate primitive pencils and non-primitive ones on $X$ which are not induced by the double covering $f : X \to Y$: A complete base-point-free linear series is said to be \emph{primitive} if its dual is also base-point-free. In \S\ref{sec:primitive} we prove that a base-point-free pencil of degree $d$ which is not induced by $f$ is primitive if $d \le g_x - 2g_y - 1 + \gon(Y)$; we then get a lower bound $s_X(2) \ge g_x - 2g_y + 1 + \gon(Y)$ by a simple calculation. In order to prove the primitiveness of a given base-point-free pencil $\linsys{D}$, we map $X$ into a certain ruled surface associated with $\linsys{D}$ and observe the possible base locus of the dual series $\linsys{K_X - D}$ by means of the generalized adjunction formula on the ruled surface.

On the other hand, for any given smooth curve $Y$ of genus $g_y \ge 2$ and an integer $g_x$ with some mind condition on $g_x$ and $g_y$, we construct a double cover $X$ of $Y$ with genus $g_x$ admitting a non-primitive complete pencil $\linsys{D}$ of degree $g_x - 2g_y + \gon(Y)$. Adding a base point $Q$ of $\linsys{K_X - D}$ to $\linsys{D}$, we obtain a base-point-free net $\linsys{D+Q}$ which gives the plane model of $X$ with the minimal possible degree $g_x - 2g_y + 1 + \gon(Y)$.

In the final section, we investigate simple nets of minimal degree on double covers whose base curves are hyperelliptic.


\subsection*{Notations and Conventions}

A \emph{smooth curve} will be algebraic and irreducible unless otherwise stated. A smooth curve $X$ is called a \emph{double cover} of a smooth curve $Y$ if there exists a \emph{double covering} $f : X \to Y$, that is, a finite flat morphism of degree $2$. A base-point-free linear series $g_{d}^{r}$ on $X$ is said to be \emph{induced by $f$} if there exists a linear series $g_{d/2}^{r}(Y)$ on $Y$ such that $g_{d}^{r} = f^{\ast}{g_{d/2}^{r}(Y)}$. A base-point-free linear series $\linsys{D}$ is said to be \emph{simple} if the morphism $\phi_{\linsys{D}} : X \to \PP^r$ associated with $\linsys{D}$ is birational onto its image. The \emph{gonality} of $Y$, which is denoted by $\gon(Y)$, is defined by
\begin{equation*}
\gon(Y) = \min \{ d : \text{there exists a base-point-free pencil $g_{d}^{1}$ on $Y$}\}.
\end{equation*}
For a ruled surface $\PP = \PP(\sheaf{O_Y} \oplus \sheaf{O_Y}(-Z))$ over $Y$, the morphism $\pi_{Z} : \PP \to Y$ denotes the \emph{projection} and a section $S_{Z}$ denotes a \emph{minimal degree section} of $\PP$ whose self-intersection number is minimal among sections of $\PP$. For a smooth variety $W$, the divisor $K_W$ of $W$ denotes a canonical divisor of $W$.


\subsubsection*{Acknowledgements} The author would like thank Prof.~Changho Keem for suggesting this problem and giving valuable comments. The author would also like thank the referee for pointing out some errors in the previous version of the paper and for providing numerous comments which improved the article.


\section{Preliminaries}
Throughout this paper, let $X$ be a smooth curve of genus $g_x$ and we always assume that there is a double covering $f : X \to Y$, where $Y$ is a smooth curve of genus $g_y \ge 2$. In this section we collect some useful facts which we make use of quite often. Let $\PP = \PP(\sheaf{O_Y} \oplus \sheaf{O_Y}(-Z))$ be a ruled surface over $Y$ where $Z \in \divisorgroup(Y)$ is an effective divisor and let $h : X \to \PP$ be a morphism which is birational onto its image such that $\pi_Z \circ h = f$.

\subsection{Double covers}
It is well known that
\begin{equation}\label{equation:f_*(O_X)}
f_{\ast}{\sheaf{O_X}} = \sheaf{O_Y} \oplus \sheaf{O_Y}(-B)
\end{equation}
for some effective divisor $B \in \divisorgroup(Y)$ with
\begin{equation}\label{equation:deg(B)}
\deg{B} = g_x - 2g_y + 1.
\end{equation}
The branch divisor of $f$ is linearly equivalent to $2B$; cf.~Hartshorne~\cite[IV, Ex.2.6(d)]{Hartshorne-1977}. Let $M \in \divisorgroup(Y)$. By Hartshorne~\cite[III, Ex.4.1]{Hartshorne-1977} and the projection formula, we have
\begin{equation}\label{equation:H^0(X, f^*(M))}
\begin{split}
H^0(X, \sheaf{O_X}(f^{\ast}{M}))
&\cong H^0(Y, f_{\ast}(\sheaf{O_X} \otimes f^{\ast}{\sheaf{O_Y}(M)})) \\
&\cong H^0(Y, (\sheaf{O_Y} \oplus \sheaf{O_Y}(-B)) \otimes \sheaf{O_Y}(M)) \\
&\cong H^0(Y, \sheaf{O_Y}(M) \oplus \sheaf{O_Y}(M-B)).
\end{split}
\end{equation}
By the Hurwitz relation (cf.~Persson~\cite[2.1]{Persson-1978}),
\begin{equation}\label{equation:K_X-f^*K_Y+f^*B}
K_X \lineqv f^{\ast}{K_Y} + f^{\ast}{B}.
\end{equation}

\begin{lemma}\label{lemma:|f^*Z|=|h^*T|}
Suppose that $h(X) \lineqv 2S_Z + \pi_Z^{\ast}(2Z)$. If $Z = B + Z'$ for some effective divisor $Z' \in \divisorgroup(Y)$ with $h^0(Y, \sheaf{O_Y}(Z')) = 1$, then
\begin{equation*}
\linsys{f^{\ast}{Z}} = \{ h^{\ast}{T} : T \in \linsys{S_Z + \pi_Z^{\ast}{Z}}\}.
\end{equation*}
\end{lemma}

\begin{proof}
It is clear that the restriction morphism
\begin{equation*}
h^{\ast} : H^0(\PP, \sheaf{O_{\PP}}(S_Z + \pi_{Z}^{\ast}{Z})) \to H^0(X, \sheaf{O_X}(f^{\ast}{Z}))
\end{equation*}
is injective. By \eqref{equation:H^0(X, f^*(M))}, we have
\begin{equation}\label{equation:H^0(X, f^*(Z))}
H^0(X, \sheaf{O_X}(f^{\ast}{Z})) \cong H^0(Y, \sheaf{O_Y}(Z) \oplus \sheaf{O_Y}(Z'))
\end{equation}
On the other hand, by Hartshorne~\cite[V, 2.4]{Hartshorne-1977} and the projection formula, we have
\begin{equation}\label{equation:H^0(PP,S_Z+pi^*Z)}
\begin{split}
H^0(\PP, \sheaf{O_{\PP}}(S_Z + \pi_{Z}^{\ast}{Z})) &\cong H^0(Y, (\pi_Z)_{\ast}(\sheaf{O_{\PP}}(S_Z) \otimes \pi_Z^{\ast}{\sheaf{O_Y}(Z)})) \\ %
&\cong H^0(Y, (\sheaf{O_Y} \oplus \sheaf{O_Y}(-Z)) \otimes \sheaf{O_Y}(Z)) \\ %
&\cong H^0(Y, \sheaf{O_Y}(Z) \oplus \sheaf{O_Y})
\end{split}
\end{equation}
Since $h^0(Y, \sheaf{O_Y}(Z'))=1$ by the hypothesis, we have
\begin{equation*}
H^0(Y, \sheaf{O_Y}(Z) \oplus \sheaf{O_Y}) = H^0(Y, \sheaf{O_Y}(Z) \oplus \sheaf{O_Y}(Z')).
\end{equation*}
Hence there is an isomorphism by \eqref{equation:H^0(X, f^*(Z))} and \eqref{equation:H^0(PP,S_Z+pi^*Z)}:
\begin{equation}\label{equation:commutative-diagram}
h^{\ast} : H^0(\PP, \sheaf{O_{\PP}}(S_Z + \pi_{Z}^{\ast}{Z})) \xrightarrow{\cong} H^0(X, \sheaf{O_X}(f^{\ast}{Z})).
\end{equation}
Therefore $\linsys{f^{\ast}{Z}} = \{ h^{\ast}{T} : T \in \linsys{S_Z + \pi_{Z}^{\ast}{Z}} \}$.
%
%
\end{proof}

\begin{lemma}[Castelnuovo-Severi inequality, {cf.~ACGH~\cite[VIII,\,Ex.\,C-1]{ACGH-1985}}]\label{lemma:CS-inequality}
Any complete base-point-free pencil on $X$ of degree $\le g_x - 2g_y$ is induced by $f$.
\end{lemma}

\subsection{Elementary transformations}
We recall first the definition of an elementary transformation; cf.~Hartshorne~\cite[V, 5.7.1]{Hartshorne-1977}. Let $P \in \PP$ and let $F = \pi_Z^{-1}(\pi_Z(P))$ be the fiber over $p=\pi_Z(P)$. Let $\blup : \imageofblup{\PP} \to \PP$ be the blowing-up of $\PP$ with center $P$. Let $\imageofblup{F}$ the strict transform of $F$ under $\blup$ and let $\bldown : \imageofblup{\PP} \to \PP'$ be the contraction of $\imageofblup{F}$. Then $\PP'$ is another ruled surface over $Y$. We denote the birational map $\bldown \circ \blup^{-1} : \PP \dashrightarrow \PP'$ by $\elm_P$ and call it \emph{the elementary transform with center $P$}. We denote the surface $\PP'$ by $\elm_P(\PP)$. For $P' = \bldown(\imageofblup{F})$, $\elm_{P'}$ is the inverse of $\elm_{P}$.

Let $P_1, \dotsc, P_n$ be distinct points in $\PP$ such that $\pi_Z(P_i) \neq \pi_Z(P_j)$ for $i \neq j$. We inductively define $\elm_{P_1, \dotsc, P_n}$ by
\begin{equation*}
\elm_{P_1, \dotsc, P_n} = \elm_{P_n} \circ \elm_{P_1, \dotsc,
P_{n-1}},
\end{equation*}
where $P_{n}$ is considered as the point $\elm_{P_1, \dotsc, P_{n-1}}(P_n) \in \elm_{P_1, \dotsc, P_{n-1}}(\PP)$. Note that $\elm_{P, Q} = \elm_{Q, P}$ if $\pi_Z(P) \neq \pi_Z(Q)$.

Let $S$ be a section of the ruled surface $\PP$ and let $E$ be an effective divisor on $S$. We define $\elm_{E}$ as follows: Let $P \in S$. Then the strict transform $\elm_{P}(S)$ of $S$ is a section of $\elm_{P}(\PP)$. We define $\elm_{nP}$ ($n \in \mathbb{N}$) by
\begin{equation*}
\elm_{nP} = \elm_{P_1, P_2, \dotsc, P_n},
\end{equation*}
where $P_1 = P$ and $P_{m+1}$ ($m=1, \dotsc, n-1$) is the unique point in $\elm_{mP}(S) \cap \pi^{-1}(\pi_Z(P))$, where $\pi : \elm_{mP}(\PP) \to Y$ is the projection. Let $E = \sum_{i=1}^{t} n_i P_i$. We define $\elm_{E}$ by $\elm_{n_1 P_1} \circ \dotsb \circ \elm_{n_t P_t}$.

Suppose that $h(X) \lineqv 2S_Z + \pi_Z^{\ast}{G}$ for some $G \in \divisorgroup(Y)$. Let $T \in \linsys{S_Z + \pi_Z^{\ast}{Z}}$ be a section of $\PP$ and let $P \in S_Z \cup T$. Let $\imageofblup{S_Z}$ and $\imageofblup{T}$ be the strict transforms of $S_Z$ and $T$ under the blowing-up $\blup$ with center $P$, respectively. Set $S_Z' = \bldown(\imageofblup{S_Z})$ and $T' = \bldown(\imageofblup{T})$, where $\psi : \widetilde{\PP} \to \PP'$ is the contraction of $\imageofblup{F}$. Since $S_Z \cap T = \varnothing$, we have $S_Z' \cap T' = \varnothing$. Therefore $\PP'$ is decomposable, i.e., $\PP' \cong \PP(\sheaf{O_Y} \oplus \sheaf{O_Y}(-Z'))$ for some $Z' \in \divisorgroup(Y)$; cf.~Hartshorne~\cite[V, Ex.2.2]{Hartshorne-1977}. Let $\imageofblup{h(X)}$ be the strict transform of $h(X)$ under $\blup$. Set $h(X)' = \bldown(\imageofblup{h(X)})$. Let $h' : X \to \elm_{P}(\PP)$ be the extension to $X$ of the map $\elm_{P} \circ h : X \dashrightarrow \elm_{P}(\PP)$. The following two lemmas are quite elementary. However we provide brief proofs for the convenience of the reader.

\begin{lemma}[{Feuntes-Pedreira~\cite{Fuentes-Pedreira-2005}, Seiler~\cite{Seiler-1992}}]\label{lemma:strict-transform-of-X}
Assume that $h(X) \cap \pi_Z^{-1}(p)$ are smooth points of $h(X)$.

\begin{enumerate}[(a)]
\item $S_Z'$ is a minimal degree section $S_{Z'}$ on $\PP(\sheaf{O_Y} \oplus \sheaf{O_Y}(-Z'))$.

\item \label{item:P-in-T-cap-h(X)} If $P \in T \cap h(X)$ and $\deg(Z-p) \ge 0$, then
\begin{align*}
&\PP' \cong \PP(\sheaf{O_Y} \oplus \sheaf{O_Y}(-(Z-p))), \\ %
&T' \lineqv S_{Z-p} + \pi_{Z-p}^{\ast}(Z-p), \\ %
&(h')^{\ast}{T'} = h^{\ast}(T) - P, \\ %
&h(X)' \lineqv 2S_{Z-p} + \pi_{Z-p}^{\ast}(G - p).
\end{align*}

\item \label{item:P-not-in-T-cap-h(X)} If $P \in T$ but $P \not\in h(X)$ and $\deg(Z-p) \ge 0$, then
\begin{align*}
&\PP' \cong \PP(\sheaf{O_Y} \oplus \sheaf{O_Y}(-(Z-p))), \\ %
&T' \lineqv S_{Z-p} + \pi_{Z-p}^{\ast}(Z-p), \\ %
&(h')^{\ast}{T'} = h^{\ast}{T} - P, \\ %
&h(X)' \lineqv 2S_{Z-p} + \pi_{Z-p}^{\ast}(G).
\end{align*}

\item \label{item:P-in-S_Z-cap-h(X)} If $P \in S_Z \cap h(X)$, then
\begin{align*}
&\PP' \cong \PP(\sheaf{O_Y} \oplus \sheaf{O_Y}(-(Z+p))), \\ %
&T' \lineqv S_{Z+p} + \pi_{Z+p}^{\ast}(Z+p), \\ %
&h(X)' \lineqv 2S_{Z+p} + \pi_{Z+p}^{\ast}(G + p).
\end{align*}

\item \label{item:P-in-S_Z-NOT-in-h(X)} If $P \in S_Z$ but $P \not\in h(X)$, then
\begin{align*}
&\PP' \cong \PP(\sheaf{O_Y} \oplus \sheaf{O_Y}(-(Z+p))), \\ %
&T' \lineqv S_{Z+p} + \pi_{Z+p}^{\ast}(Z+p), \\ %
&h(X)' \lineqv 2S_{Z+p} + \pi_{Z+p}^{\ast}(G + 2p).
\end{align*}
Furthermore, $\bldown(\imageofblup{F}) = h(X)' \cap \pi_{Z+p}^{\ast}(p)$ is a singular point of $h(X)'$.
\end{enumerate}
\end{lemma}

\begin{proof}
By Seiler~\cite[Lemma~6]{Seiler-1992}, $S_Z'$ is a minimal degree section of $\PP'$. It follows by Fuentes-Pedreira~\cite[4.12]{Fuentes-Pedreira-2005} that
\begin{equation*}
\PP' \cong
\begin{cases}
\PP(\sheaf{O_Y} \oplus \sheaf{O_Y}(-(Z-p))) & \text{if $P \in T$}, \\ %
\PP(\sheaf{O_Y} \oplus \sheaf{O_Y}(-(Z+p))) & \text{if $P \in S_Z$}.
\end{cases}
\end{equation*}
By Fuentes-Pedreira~\cite[4.6]{Fuentes-Pedreira-2005}, we have $T' \lineqv S_Z' + \pi_{Z'}^{\ast}(Z + (1-2\mu_P(T)p))$, where $\mu_{P}$ is the multiplicity of $T$ at $P$. Therefore
\begin{equation*}
T' \lineqv
\begin{cases}
S_{Z-p} + \pi_{Z-p}^{\ast}(Z-p) & \text{if $P \in T$}, \\ %
S_{Z+p} + \pi_{Z+p}^{\ast}(Z+p) & \text{if $P \not\in T$, i.e, $P \in S_Z$}.
\end{cases}
\end{equation*}

Let $\phi=\blup \circ \bldown^{-1} : \PP' \dashrightarrow \PP$ be the inverse of $\elm_P$. By Fuentes-Pedreira~\cite[4.4]{Fuentes-Pedreira-2005}, we have
\begin{equation}\label{equation:phi^*(h(X))-I}
\phi^{\ast}(h(X)) = h(X)' + \mu_{P}(h(X))\pi_{Z'}^{\ast}(p).
\end{equation}
On the other hand, for any $y \in Y$, $\phi^{\ast}(\pi_{Z}^{\ast}{y}) = \pi_{Z'}^{\ast}{y}$; Fuentes-Pedreira~\cite[4.3]{Fuentes-Pedreira-2005}. Since $h(X) \lineqv 2S_Z + \pi_Z^{\ast}{G}$, it follows that
\begin{equation}\label{equation:phi^*(h(X))-II}
\phi^{\ast}(h(X)) \lineqv 2\phi^{\ast}{S_Z} + \phi^{\ast}{\pi_{Z}^{\ast}{G}} \lineqv 2\phi^{\ast}{S_Z} + \pi_{Z'}^{\ast}{G}.
\end{equation}
Note that
\begin{equation*}
\phi^{\ast}{S_Z} \lineqv
\begin{cases}
S_{Z'} + \pi_{Z'}^{\ast}{p} & \text{if $P \in S_Z$}, \\ %
S_{Z'} & \text{if $P \not \in S_Z$, i.e., $P \in T$}.
\end{cases}
\end{equation*}
Therefore it follows by \eqref{equation:phi^*(h(X))-I} and \eqref{equation:phi^*(h(X))-II} that
\begin{equation*}
h(X)' \lineqv
\begin{cases}
2S_{Z'} + \pi_{Z'}^{\ast}(G - p) & \text{if $P \in T \cap h(X)$}, \\ %
2S_{Z'} + \pi_{Z'}^{\ast}(G + p) & \text{if $P \in S_Z \cap h(X)$}, \\ %
2S_{Z'} + \pi_{Z'}^{\ast}(G + 2p) & \text{if $P \in S_Z$ but$P \not\in h(X)$}.
\end{cases} \qedhere
\end{equation*}
\end{proof}

Let $P \in h(X)$ which is not a singular point of $h(X)$. Let $h' : X \to \elm_{P}(\PP)$ be the extension to $X$ of the map $\elm_{P} \circ h : X \dashrightarrow \elm_{P}(\PP)$.

\begin{lemma}\label{lemma:sigma^*(P)<=h^*(X)}
If $S$ is a section on $\PP$ such that $P \not\in S$, then $\sigma^{\ast}{P} \le h'^{\ast}{S'}$, where $\sigma : X \to X$ is an involution induced by $f$.
\end{lemma}

\begin{proof}
Let $\imageofblup{F}$, $\imageofblup{h(X)}$, and $\imageofblup{S}$ be the strict transforms of $F = \pi_M^{-1}(\pi_M(P))$, $h(X)$, and $S$ under the blowing-up $\blup$ at $P$, respectively. Set $h(X)' = \bldown(\imageofblup{h(X)})$ and $S'=\bldown(\imageofblup{S})$. Clearly, $F \cap h(X) = \{P, \sigma^{\ast}{P}\}$ and $F.S=1$. Since $P \in F \cap h(X)$ but $P \not\in S$ by the assumption, we have $\imageofblup{F} \cap \imageofblup{h(X)} = \{\blup^{-1}(\sigma^{\ast}{P})\}$ and $\imageofblup{F}.\imageofblup{S}=1$. Since  $\bldown(\imageofblup{F}) = \sigma^{\ast}{P}$, we have $\sigma^{\ast}{P} \in h(X)'$ and $\sigma^{\ast}{P} \in S'$. Therefore $\sigma^{\ast}{P} \le h'^{\ast}{S}$.
\end{proof}


\section{Primitive pencils on double covers}\label{sec:primitive}

In this section, we prove that certain base-point-free pencils on $X$ are primitive; Proposition~\ref{proposition:primitive}. We then get a lower bound for $s_X(2)$; Theorem~\ref{theorem:lower-bound}

Let $\linsys{D}$ be a base-point-free pencil of degree $d \le g_x - 1$ on $X$ which is not induced by $f$. In order to verify the primitiveness of $\linsys{D}$, we map $X$ into a certain ruled surface:

\begin{lemma}\label{lemma:j_M-construction}
There exist an effective divisor $M \in \divisorgroup(Y)$ of degree $d$ and a morphism
\begin{equation*}
j_M : X \to \PP(\sheaf{O_Y} \oplus \sheaf{O_Y}(-M))
\end{equation*}
which is birational onto its image such that $M \lineqv f_{\ast}{D}$, $\pi_M \circ j_M = f$, and $j_M(X) \lineqv 2S_M + \pi_M^{\ast}(2M)$.
\end{lemma}

\begin{proof}
Let $\phi_{\linsys{D}} : X \to \PP^1$ be the morphism associated with $\linsys{D}$. Since $\linsys{D}$ is not induced by $f$, the morphism
\begin{equation*}
j_0 = \phi_{\linsys{D}} \times f : X \to \PP^1 \times Y = \PP(\sheaf{O_Y}
\oplus \sheaf{O_Y})
\end{equation*}
is birational onto its image. Let $S_0 = \{a\} \times Y$ ($a \in \PP^1$) be a minimal degree section of $\PP^1 \times Y$, i.e., a fiber of the first projection $\PP^1 \times Y \to \PP^1$. Set $D_0 = j_0^{\ast}{S_0}$. We may assume that $D_0$ consists of distinct points and $j_0(\supp(D_0))$ does not contain any singular points of $X_0 := j_0(X)$. Set $M = f_{\ast}{D_0}$. Then
\begin{equation*}
X_0 \lineqv 2S_0 + \pi_0^{\ast}{M}.
\end{equation*}

Since $D_0$ is cut out on $X_0$ by a section $S_0$, we may regard $D_0$ as a divisor of $S_0$. We apply $\elm_{D_0}$ to $\PP^1 \times Y$. Since $f_{\ast}{D_0} = M$, it follows by Lemma~\ref{lemma:strict-transform-of-X}(\ref{item:P-in-S_Z-cap-h(X)}) that
\begin{equation*}
\elm_{D_0}(\PP^1 \times Y) \cong \PP(\sheaf{O_Y} \oplus \sheaf{O_Y}(-M)).
\end{equation*}
Let $j_M : X \to \PP(\sheaf{O_Y} \oplus \sheaf{O_Y}(-M))$ be the extension to $X$ of the map $\elm_{D_0} \circ j_0 : X \dashrightarrow \PP(\sheaf{O_Y} \oplus \sheaf{O_Y}(-M))$. It is clear that $j_M$ is birational onto its image and $\pi_M \circ j_M = f$. Since $\supp(D_0) \subset X_0 \cap S_0$ and $X_0 \lineqv 2S_0 + \pi_0^{\ast}{M}$,  we have
\begin{equation*}
j_M(X) \lineqv 2S_M + \pi_M^{\ast}(2M)
\end{equation*}
by Lemma~\ref{lemma:strict-transform-of-X}(\ref{item:P-in-S_Z-cap-h(X)}).
\end{proof}

Throughout this section, let $j_M : X \to \PP(\sheaf{O_Y} \oplus \sheaf{O_Y}(-M))$ be the morphism associated with $\linsys{D}$ given by Lemma~\ref{lemma:j_M-construction}. Set $\PP = \PP(\sheaf{O_Y} \oplus \sheaf{O_Y}(-M))$. The conductor $\Delta$ of the birational morphism $j_M : X \to j_M(X)$ is a pull-back of an effective divisor of $Y$ via $f^{\ast}$:

\begin{lemma}\label{lemma:j_M-conductor}
If $g_x \ge 4g_y - 2$, then $\Delta = f^{\ast}{N}$ for some $N \in \linsys{M-B}$.
\end{lemma}

\begin{proof}
Let $X_M = j_M(X)$ for convenience. By the generalized adjunction formula (cf.~Iitaka~\cite[p.205]{Iitaka-1982}), we have
\begin{equation}\label{equation:K_X-gen.adj.form.}
K_X \lineqv j_M^{\ast}((K_{\PP} + X_M)|_{X_M}) - \Delta.
\end{equation}
Since $K_{\PP} \lineqv -2S_M + \pi_M^{\ast}(-M + K_Y)$ by Hartshorne~\cite[V, 2.10]{Hartshorne-1977} and $X_M \lineqv 2S_M + \pi_M^{\ast}(2M)$ by the hypothesis, it follows by \eqref{equation:K_X-gen.adj.form.} that
\begin{equation}\label{equation:K_X=f^*(M+K_Y)-Delta}
K_X \lineqv f^{\ast}(M + K_Y) - \Delta.
\end{equation}
On the other hand, by \eqref{equation:K_X-f^*K_Y+f^*B}, $K_X \lineqv f^{\ast}{K_Y} + f^{\ast}{B}$; hence, it follows by \eqref{equation:K_X=f^*(M+K_Y)-Delta} that
\begin{equation}\label{equation:Delta-f^*(M-B)}
\Delta \lineqv f^{\ast}(M-B),
\end{equation}

Since $d \le g_x - 1$ and $g_x \ge 4g_y - 2$ by the assumptions, we have
\begin{equation*}
\deg(M-2B) \le (g_x - 1) - 2(g_x - 2g_y + 1) = -g_x + 4g_y - 3 < 0.
\end{equation*}
Hence it follows by \eqref{equation:H^0(X, f^*(M))} that
\begin{equation*}
\begin{split}
H^0(X, \sheaf{O_X}(f^{\ast}(M-B)))
&\cong H^0(Y, \sheaf{O_Y}(M-B) \oplus \sheaf{O_Y}(M-2B)) \\ %
&\cong H^0(Y, \sheaf{O_Y}(M-B)).
\end{split}
\end{equation*}
Therefore
\begin{equation}\label{equation:|f^*(M-B)|=f^*|M-B|}
\linsys{f^{\ast}(M-B)} = f^{\ast}{\linsys{M-B}}.
\end{equation}
Since $\Delta \in \linsys{f^{\ast}(M-B)}$ by \eqref{equation:Delta-f^*(M-B)}, we have
\begin{equation*}
\Delta \in f^{\ast}{\linsys{M-B}}
\end{equation*}
by \eqref{equation:|f^*(M-B)|=f^*|M-B|}. Therefore  $\Delta = f^{\ast}{N}$ for some $N \in \linsys{M-B}$.
\end{proof}

There is a special configuration of divisors on $X$ induced by a section of $\PP$:

\begin{lemma}\label{lemma:j_M-configuration}
Assume that $g_x \ge 4g_y - 2$. If $d \le g_x - 2g_y - 1 + \gon(Y)$, then there exist an effective divisor $D' \in \divisorgroup{X}$ and a section $T$ of $\PP(\sheaf{O_Y} \oplus \sheaf{O_Y}(-M))$ satisfying the followings:
\begin{enumerate}[(a)]
\item \label{item:supp(f_*(D))-cap-supp(f_*(D'))=0} $\supp(f_{\ast}{D}) \cap \supp(f_{\ast}{D'}) = \varnothing$, i.e., no two points of $\supp(D)$ and $\supp(D')$ lie over the same fiber.

\item \label{item:f_*(D)-in-|M|} $f_{\ast}{D}, f_{\ast}{D'} \in \linsys{M}$,

\item \label{item:j_M^*(T)=D+D'} $j_M^{\ast}{T} = D + D' \in \linsys{f^{\ast}{M}}$.
\end{enumerate}
\end{lemma}

\begin{proof}
We proceed as in the proof of Lemma~\ref{lemma:j_M-construction}: Let $j_0 = \phi_{\linsys{D}} \times f : X \to \PP^1 \times Y$, $X_0=j_0(X)$, $S_0 = \{a\} \times Y$ ($a \in \PP^1$), $D_0 = j_0^{\ast}{S_0}$, $M= f_{\ast}{D_0}$ be as defined there. Since $D$ also lies on a fiber of the first projection $\PP^1 \times Y \to \PP^1$, $D = j_0^{\ast}{S}$ for some $S = \{b\} \times Y$ ($b \in \PP^1$). We assumed that $D_0$ consists of distinct points and $j_0(\supp(D_0))$ does not contain any singular points of $X_0$; but, we may further assume that
\begin{equation}\label{equation:supp(D)-cap-supp(D_0)=0}
\supp(f_{\ast}{D}) \cap \supp(f_{\ast}{D_0}) = \varnothing,
\end{equation}
i.e., no two points of $\supp(D)$ and $\supp(D_0)$ lie over the same fiber. Since $M=f_{\ast}{D_0}$, we have
\begin{equation}\label{equation:f_*(D)-in-|M|}
f_{\ast}{D}, f_{\ast}{D_0} \in \linsys{M}.
\end{equation}
Applying $\elm_{D_0}$ to $\PP^1 \times Y$, we got the morphism $j_M : X \to \PP$.

Set $T = \elm_{D_0}(S)$ and $D' = \sigma^{\ast}{D_0}$, where $\sigma : X \to X$ is the involution induced by $f$. It is clear that $\supp(f_{\ast}{D}) \cap \supp(f_{\ast}{D'}) = \varnothing$ by \eqref{equation:supp(D)-cap-supp(D_0)=0}. Note that $T$ is again a section of $\PP$.

\claiminline $j_M^{\ast}{T} = D + D'$: For any $P \in \supp(D_0)$, applying $\elm_{P}$, we have
\begin{equation*}
\sigma^{\ast}{P} \in j_0'^{\ast}(\elm_{P}(S))
\end{equation*}
by Lemma~\ref{lemma:sigma^*(P)<=h^*(X)}, where  $j_0' = \elm_P \circ j_0$. Since $T = \elm_{D_0}(S)$, it follows that $D' = \sigma^{\ast}{D_0} \le j_M^{\ast}{T}$. On the other hand, $\elm_{D_0}$ does not affect $S$ outside $\pi_0^{\ast}(f_{\ast}{D_0})$ and $D$ lies on $S$ outside $\pi_0^{\ast}(f_{\ast}{D_0})$ by \eqref{equation:supp(D)-cap-supp(D_0)=0}. Hence
\begin{equation*}
D = j_0^{\ast}{S} \le j_M^{\ast}{\elm_{D_0}(S)} = j_M^{\ast}{T}.
\end{equation*}
Therefore $j_M^{\ast}{T} = D + D'$. Since $j_M^{\ast}{T} \lineqv S_M + \pi_M^{\ast}{M}$ by Lemma~\ref{lemma:strict-transform-of-X}(d), we have $j_M^{\ast}{T} \in \linsys{f^{\ast}{M}}$.

\end{proof}

\begin{proposition}\label{proposition:primitive}
Assume that $g_x \ge 4g_y - 2$. Let $\linsys{D}$ be a complete base-point-free pencil of degree $d$ on $X$ which is not induced by $f$. If
\begin{equation*}
d \le g_x - 2g_y - 1 + \gon(Y),
\end{equation*}
then $\linsys{D}$ is primitive.
\end{proposition}

\begin{proof}
We need to prove that $\linsys{K_X - D}$ is base-point-free. Set $X_M = j_M(X)$. By  Lemma~\ref{lemma:j_M-configuration}, there exist an effective divisor $D' \in \divisorgroup{X}$ and a section $T$ of $\PP$ satisfying three conditions in Lemma~\ref{lemma:j_M-configuration}. Let $\Delta$ be the conductor of $j_M$. By \eqref{equation:K_X=f^*(M+K_Y)-Delta}, we have $K_X \lineqv f^{\ast}(M+K_Y)-\Delta$. On the other hand, $\Delta = f^{\ast}{N}$ for some $N \in \linsys{M-B}$ by Lemma~\ref{lemma:j_M-conductor}; hence it follows that
\begin{equation*}
\begin{split}
\linsys{K_X - D} &= \linsys{f^{\ast}(M + K_Y) - \Delta - D} \\ %
&= \linsys{f^{\ast}(M+K_Y) - f^{\ast}{N} - D} \\ %
&= \linsys{(f^{\ast}{M} - D) + f^{\ast}(K_Y-N)};
\end{split}
\end{equation*}
hence
\begin{equation}
\linsys{f^{\ast}M - D} + \linsys{f^{\ast}(K_Y - N)} \subset \linsys{K_X - D}.
\end{equation}
Therefore it is enough to show that $\linsys{f^{\ast}M - D}$ and $\linsys{f^{\ast}(K_Y - N)}$ are base-point-free.

Since $f_{\ast}{D} \in \linsys{M}$  by Lemma~\ref{lemma:j_M-configuration}(\ref{item:f_*(D)-in-|M|}), we have
\begin{equation*}
\sigma^{\ast}{D} = f^{\ast}{f_{\ast}{D}} - D \in \linsys{f^{\ast}{M} - D},
\end{equation*}
where $\sigma : X \to X$ is the involution induced by $f$. On the other hand, $j_M^{\ast}{T} \in \linsys{f^{\ast}{M}}$ and $j_M^{\ast}{T} = D + D'$ by Lemma~\ref{lemma:j_M-configuration}(\ref{item:j_M^*(T)=D+D'}); hence
\begin{equation*}
D' = j_M^{\ast}{T} - D \in \linsys{f^{\ast}{M} - D}.
\end{equation*}
No two points of $\supp(D)$ and $\supp(D')$ lie over the same fiber by Lemma~\ref{lemma:j_M-configuration}(\ref{item:supp(f_*(D))-cap-supp(f_*(D'))=0}); hence
\begin{equation*}
\supp(\sigma^{\ast}{D}) \cap \supp(D') = \varnothing.
\end{equation*}
Therefore $\linsys{f^{\ast}M - D}$ is base-point-free.

\claiminline $\linsys{f^{\ast}(K_Y-N)}$ is base-point-free: It is enough to show that $\linsys{K_Y - N}$ is base-point-free. First of all, since $\linsys{D}$ is not induced by $f$, by the Castelnuovo-Severi inequality(Lemma~\ref{lemma:CS-inequality}), we have
\begin{equation*}
g_x -2 g_y + 1 \le \deg{M}.
\end{equation*}
Since $\deg{B} = g_x - 2g_y + 1$ and $d \le g_x -2g_y - 1 + \gon(Y)$ by the hypothesis, we have
\begin{equation*}
0 \le \deg{N} = \deg(M-B) \le \gon(Y) - 2.
\end{equation*}
Therefore
\begin{equation*}
h^0(Y, \sheaf{O_Y}(N)) = h^0(Y, \sheaf{O_Y}(N+y)) = 1
\end{equation*}
for all $y \in Y$, which implies that $\linsys{K_Y - N}$ is base-point-free.
\end{proof}

We get a lower bound for $s_X(2)$.

\begin{lemma}[{Coppens-Keem-Martens~\cite[2.2.1]{Coppens-Keem-Martens-1992}}]\label{lemma:subtracting=>not-induced}
Let $h : C \to C'$ be a non-trivial covering of curves and $g_{d}^{r}$ ($r \ge 1$) a base-point-free linear series on $C$ which is not induced by $h$. Let $P_1, \dotsc, P_{r-1}$ be $r-1$ general points of $C$. Then the base-point-free part of the linear series $\linsys{g_{d}^{r} - P_1 - \dotsb - P_{r-1}}$ on $C$ is not induced by $h$.
\end{lemma}

\begin{theorem}\label{theorem:lower-bound}
Let $X$ be a smooth curve of genus $g_x$. Assume that $X$ is a double cover of a smooth curve $Y$ of genus $g_y \ge 2$. If $g_x \ge 4g_y - 2$, then
\begin{equation*}
s_{X}(2) \ge g_x - 2g_y + 1 + \gon(Y).
\end{equation*}
\end{theorem}

\begin{proof}
Let $g_{s_{X}(2)}^2$ be a simple net on $X$. Subtract a general point $P \in X$ from $g_{s_{X}(2)}^2$; we get a complete base-point-free pencil
\begin{equation*}
g_{s_{X}(2) - 1}^{1} = \linsys{g_{s_{X}(2)}^{2} - P}.
\end{equation*}
By Lemma~\ref{lemma:subtracting=>not-induced}, $g_{s_{X}(2) - 1}^{1}$ is not induced by $f$. On the other hand $g_{s_{X}(2) - 1}^{1}$ cannot be primitive; cf.~Coppens-Keem-Martens~\cite[1.1.3]{Coppens-Keem-Martens-1992}. Therefore, applying Proposition~\ref{proposition:primitive} to $g_{s_X(2)-1}^1$, we have
\begin{equation*}
s_{X}(2) - 1 \ge g_x - 2g_y + \gon(Y). \qedhere
\end{equation*}
\end{proof}

\begin{proposition}\label{lemma:phi*psi-is-birational}
Assume that $g_x \ge 8g_y + 2$. Let $\psi_k : Y \to \PP^1$ be a morphism of degree $\gon(Y)$ and let $\psi = f \circ \psi_k$. Then there exists a morphism $\phi : X \to \PP^1$ of degree $g_x - 2g_y + 1$ such that the morphism $\phi \times \psi : X \to \PP^1 \times \PP^1$ is birational onto its image.
\end{proposition}

\begin{proof}
If $g_x - 2g_y + 1$ and $2\gon(Y)$ are relatively prime, the theorem is clear. We now assume that $g_x - 2g_y + 1$ and $2\gon(Y)$ are \emph{not} relatively prime. There is a morphism $\phi : X \to \PP^1$, say, of degree $g_x - 2g_y + 1$ which does not factor through $f$; Keem-Ohbuchi~\cite{Keem-Ohbuchi-2004}. Suppose that the morphism $\phi \times \psi$ is not birational onto its image. Then $\phi : X \to \PP^1$ is induced by a morphism $h : X \to Z$ of degree $n \ge 2$ where $Z$ is a smooth curve of genus $g_z < g_x$. We then have the following commutative diagram
\begin{equation*}
\xymatrix{%
 & & X \ar[d]_{h} \ar[ddll]_{\phi} \ar[dr]^{f} & & \\ %
 & & Z \ar[dll] \ar[drr] & Y \ar[dr]^{\psi_k}& \\ %
\PP^1 & & & & \PP^1
}
\end{equation*}
We now prove that $g_z \ge 3$. Suppose that $g_z \le 2$. Since $\phi$ does not factor through $f$, the morphism $h$ cannot factor through the double covering  $f$. Therefore it follows by the Castelnuovo-Severi inequality that
\begin{equation}\label{equation:g_x<=ng_z+2g_y+n-1}
g_x \le ng_z + 2g_y + n-1.
\end{equation}
Since $n \le \gon(Y) \le \frac{g_y + 3}{2}$, we then have
\begin{equation*}
g_x \le \frac{g_y + 3}{2} \cdot 2 + 2g_y + \frac{g_y + 3}{2} - 1 = \frac{7g_y + 7}{2}
\end{equation*}
from \eqref{equation:g_x<=ng_z+2g_y+n-1}, which contradicts to the hypothesis $g_x \ge 8g_y + 2$. Therefore $g_z \ge 3$.

We now prove that there exists a base-point-free pencil $g_{g_x - 2g_y + 1}^{1}$ on $X$ such that it is not induced by any morphism $h : X \to Z$. Let
\begin{equation*}
\deFranchis = \left\{ (Z, h) \left| %
\begin{aligned}
&\text{$Z$ is a smooth curve of genus $g_z \ge 3$,} \\ %
&\text{$h : X \to Z$ is a covering of degree $n$,} \\ %
&\text{where $n \ge 2$ is a common divisor of $g_x - 2g_y + 1$ and $2\gon(Y)$.}
\end{aligned}\right.\right\}
\end{equation*}
Note that $\deFranchis$ is a finite set by de Franchis' theorem. Fix $(Z, h) \in \deFranchis$.
Let $\Sigma$ be an irreducible component of $W_{g_x - 2g_y + 1}^{1}(X)$ whose general element is a base-point-free complete pencil not induced by the double covering $f$. Since $g_x \ge 8g_y - 3$, it follows by Ballico-Keem~\cite[Proposition~3.1]{Ballico-Keem-2006} that
\begin{equation*}
\dim{\Sigma} = g_x - 4g_y.
\end{equation*}
On the other hand, we have
\begin{equation*}
\dim{h^{\ast}{W_{\frac{g_x - 2g_y + 1}{n}}^{1}(Z)}} \le g_z.
\end{equation*}
By Riemann-Hurwitz formula, we have
\begin{equation*}
g_z \le \frac{g_x + n - 1}{n}.
\end{equation*}
Since $g_x \ge 8g_y + 2$ and $n \ge 2$, we have
\begin{equation*}
\begin{split}
&\dim{\Sigma} - \dim{h^{\ast}{W_{\frac{g_x - 2g_y + 1}{n}}^{1}(Z)}} \\ %
&\quad\quad \ge g_x - 4g_y - g_z  \\ %
&\quad\quad \ge g_x - 4g_y - \frac{g_x + n - 1}{n} = \frac{(n-1)g_x - 4ng_y - n + 1}{n} \\ %
&\quad\quad \ge \frac{(n-1)(8g_y + 2) - 4ng_y - n + 1}{n} = \frac{4ng_y - 8g_y + n - 1}{n} \\ %
&\quad\quad> 0
\end{split}
\end{equation*}
Therefore one may take a general pencil $g_{g_x - 2g_y + 1}^{1}$ of $X$ so that it is not induced by $f$ and
\begin{equation*}
g_{g_x - 2g_y + 1}^{1} \in \Sigma \setminus \left(\bigcup_{(Z, h) \in \deFranchis} h^{\ast}{W_{\frac{g_x - 2g_y + 1}{n}}^{1}(Z)}\right). \qedhere
\end{equation*}
\end{proof}

We finally get an upper bound for $s_X(2)$.

\begin{theorem}\label{theorem:upper-bound}
Let $X$ be a smooth curve of genus $g_x$. Assume that $X$ is a double cover of a smooth curve $Y$ of genus $g_y \ge 2$. If $g_x \ge 8g_y + 2$, then
\begin{equation*}
s_{X}(2) \le g_x - 2g_y - 1 + 2\gon(Y).
\end{equation*}
\end{theorem}

\begin{proof}
From Proposition~\ref{lemma:phi*psi-is-birational} there is a model $\Gamma$ of $X$ on $\PP^1 \times \PP^1$ of bidegree $(g_x - 2g_y + 1, 2 \gon(Y))$. Since $g_x \ge 8g_y + 2$, it follows that
\begin{equation*}
(g_x - 2g_y)(2\gon(Y)-1) > g_x.
\end{equation*}
Therefore the model $\Gamma$ has a singular point $s$.

Embed $\PP^1 \times \PP^1$ as a smooth quartic $Q$ in $\PP^3$. Then $\Gamma$ has degree $g_x - 2g_y + 1 + 2\gon(Y)$ and the projection with center $s$ gives rise to a birational equivalence between $Q$ and $\PP^2$. The closure of the image of $\Gamma$ is a plane model of $X$ of degree $g_x - 2g_y - 1 + 2\gon(Y)$; hence $s_{X}(2) \le g_x - 2g_y - 1 + 2\gon(Y)$.
\end{proof}

\begin{remark}\label{remark:hyperelliptic-case}\hfill
\begin{enumerate}[(a)]
\item In case $Y$ is hyperelliptic, we have $s_{X}(2) = g_x - 2g_y +3$ if $g_x \ge 8g_y + 2$.

\item It would be an intriguing problem finding an example where $s_{X}(2) > g_x - 2g_y + 1 + \gon(Y)$.
\end{enumerate}
\end{remark}


\section{Double covers with plane models of minimal possible degree}\label{sec:minimal-possible}

Let $Y$ be a smooth curve of genus $g_y \ge 2$. In this section, we construct a double cover $X$ of $Y$ with genus $g_x$ which has a plane model of the minimal possible degree $g_x - 2g_y + 1 + \mathrm{gon}(Y)$.

\begin{theorem}\label{theorem:example}
Let $Y$ be a smooth curve of genus $g_y \ge 2$. For any integer $g_x$ with $g_x \ge 4g_y - 1$, there exists a smooth double cover $X$ of $Y$ with genus $g_x$ such that $s_X(2) = g_x - 2g_y + 1 + \gon(Y)$.
\end{theorem}

\begin{proof}
Let $k = \gon(Y)$. Choose a divisor $p_1 + \dotsb + p_k \in g_k^1(Y)$ consisting of different points. Set $L = p_1 +\dotsb + p_{k-1}$ and $N = p_k + q_1 + \dotsb q_{d'}$, where $q_1, \dotsc, q_{d'}$ ($d' = g_x - 2g_y + 1 - k$) are any different points in $Y$. Set $M = 2L + N$.

Let $S_L$ be the minimal degree section of a ruled surface $\pi_L : \PP(\sheaf{O_Y} \oplus \sheaf{O_Y}(-L)) \to Y$. Let $P_i$ ($i=1, \dotsc, k$) and $Q_j$ ($j=1, \dotsc, d'$) be the points on $S_L$ such that $\pi_L(P_i) = p_i$ and $\pi_L(Q_j) = q_j$, respectively. On the ruled surface $\PP(\sheaf{O_Y} \oplus \sheaf{O_Y}(-L))$ we define a linear system
\begin{equation*}
\mathbb{X} = \{ X \in \linsys{2S_L + \pi_L^{\ast}{M}} : \text{$X$ passes through $P_k, Q_1, \dotsc, Q_{d'}$} \}.
\end{equation*}

We have two sublinear systems of $\mathbb{X}$:
\begin{align*}
\mathbb{H} &= \{ S_L + H : H \in \linsys{S_L + \pi_L^{\ast}{M}} \}, \\
\mathbb{T} &= \{ T_1 + T_2 + \pi_L^{\ast}{N} : T_1, T_2 \in \linsys{S_L + \pi_L^{\ast}{L}} \}.
\end{align*}
By the assumption $g_x \ge 4g_y - 1$, the linear system $\linsys{S_L + \pi_L^{\ast}{M}}$ is base-point-free. We may choose an element $H \in \linsys{S_L + \pi_L^{\ast}{M}}$ so that $H$ does not pass through any of the points $P_k, Q_1, \dotsc, Q_{d'}$. We may also choose $T_i$ ($i=1,2$) so that $T_i \cap S_L = \varnothing$ ($i=1,2$). Therefore $P_k, Q_1, \dotsc, Q_{d'}$ are the only base points of the linear system $\mathbb{X}$. By the Bertini theorem of characteristic zero, a general member of $\mathbb{X}$ can have singularities only at the base points $P_k, Q_1, \dotsc, Q_{d'}$. Let $X_0$ be a general member of $\mathbb{X}$. Since $X_0.S_L = (2S_L + \pi_L^{\ast}{M}).S_L = d'+1$, it follows that $X_0$ cannot have singularities at any points of the base points $P_k, Q_1, \dotsc, Q_{d'}$, which implies that $X_0$ is irreducible and smooth. Choose a section $S$ of $\PP(\sheaf{O_Y} \oplus \sheaf{O_Y}(-L))$ such that $S \cap S_L = \varnothing$. Let $P_i'$ ($i=1, \dotsc, k-1$) be the points on $S$ such that $\pi_L(P_i') = p_i$. Since $P_i$ and $P_i'$ ($i=1, \dotsc, k-1$) are not base points of $\mathbb{X}$, we may assume that $X_0$ does not pass through any of $P_i, P_i'$.

Let $L' = P_1'+\dotsb+P_{k-1}'$. Set $X = \elm_{L'}(X_0)$ and $S_0 = \elm_{L'}(S_L)$. By Lemma~\ref{lemma:strict-transform-of-X}(\ref{item:P-not-in-T-cap-h(X)}), $\elm_{L'}(\PP(\sheaf{O_Y} \oplus \sheaf{O_Y}(-L)))= \PP(\sheaf{O_Y} \oplus \sheaf{O_Y}) = Y \times \PP^1$ and $X \lineqv 2S_0 + \pi_0^{\ast}{M}$. Hence the restriction $f = \pi_0|_{X} : X \to Y$ of the first projection $\pi_0 : Y \times \PP^1 \to Y$ is a double covering. Since $\supp(L') \cap X = \varnothing$, $X$ has singular points at $P_1, \dotsc, P_{k-1}$. Therefore $S_0|_X = f^{\ast}(p_1+\dotsb+p_{k-1}) + P_k + Q_1 + \dotsb + Q_{d'}$. Since $S_0$ is a fiber of the second projection $Y \times \PP^1 \to \PP^1$, the linear series $\linsys{D} = \linsys{f^{\ast}(p_1+\dotsb+p_{k-1}) + P_k + Q_1 + \dotsb + Q_{d'}}$ is base-point-free.

Let $\overline{P_k}$ be the conjugate point of $P_k$ with respect to the double covering $f : X \to Y$. Set $\linsys{E} = \linsys{D + \overline{P_k}} = \linsys{f^{\ast}{g_k^1(Y)} + Q_1 + \dotsb + Q_{d'}}$. Since $Q_1 + \dotsb + Q_{d'}$ are not base points of $\linsys{E}$, we have $\dim{\linsys{E}} > \dim{\linsys{D}}$. Therefore $\linsys{E}$ is a base-point-free simple net on $X$ of degree $g_x - 2g_y + 1 + k$.
\end{proof}


\section{Double covers of hyperelliptic curves}\label{sec:hyperelliptic}

In section~\ref{sec:minimal-possible}, we constructed a double cover with a simple net of the form $\linsys{f^{\ast}{g_k^1(Y)} + Q_1 + \dotsb + Q_{d'}}$ ($d' = g_x - 2g_y + 1 - k$). Assume that $Y$ is hyperelliptic with $g_y \ge 2$. Then, according to Remark~\ref{remark:hyperelliptic-case}(a), we have $s_X(2) = g_x - 2g_y + 3$. In this section we prove that every simple net of degree $s_X(2) = g_x - 2g_y + 3$ on $X$ is of the form
\begin{equation}\label{equation:of-the-form}
\linsys{f^{\ast}{g_2^1(Y)} + Q_1 + \dotsb + Q_{d'}}
\end{equation}
where $d' = g_x - 2g_y - 1$.

\begin{lemma}[{Accola~\cite[5.1]{Accola-1979}}]\label{lemma:dim(g_l^t+g_m^s)}
Suppose $g_{l}^{t}$ is simple and $g_{m}^{s}$ is a different linear series, possibly composite, so that $t \ge s$. Then $\dim \linsys{g_{l}^{t} + g_{m}^{s}} \ge t + 2s$.
\end{lemma}

\begin{lemma}[{Accola~\cite[5.3]{Accola-1979}}]\label{lemma:dim(g_l^t-g_m^1)}
Let $g_l^t$ and $g_m^1$ be base-point-free linear series. Assume that $g_l^t$ is simple. If $\dim \linsys{g_l^t + ng_m^1} = t + 2n$, then $\dim{\linsys{g_l^t -g_m^1}} = t - 2$.
\end{lemma}

\begin{theorem}\label{theorem:main-theorem-hyperelliptic-case}
Let $X$ be a smooth curve of genus $g_x$ which is a double cover of a hyperelliptic curve $Y$ of genus $g_y \ge 2$. Then every simple net $\linsys{D}$ of degree $s_X(2)(= g_x - 2g_y + 3)$  on $X$ is of the form
\begin{equation*}
\linsys{f^{\ast}{g_2^1(Y)} + Q_1 + \dotsb + Q_{d'}}
\end{equation*}
for some $R_i \in X$, where $g_2^1(Y)$ is a base-point-free pencil on $Y$.
\end{theorem}

\begin{proof}
Set
\begin{equation*}
g_{g_x + 2g_y - 1}^r = \linsys{D + f^{\ast}{K_Y}} = \linsys{D + (g_y - 1)f^{\ast}{g_2^1(Y)}}.
\end{equation*}

\noindent Claim. $r=2g_y$: For each $i=0, \dotsc, g_y-1$, set $g_{l(i)}^{t(i)} = \linsys{D + if^{\ast}{g_2^1(Y)}}$ and $g_4^1 = f^{\ast}{g_2^1(Y)}$. Note that $g_{l(i)}^{t(i)}$ is simple. By Lemma~\ref{lemma:dim(g_l^t+g_m^s)},
\begin{equation*}
t(i+1) \ge t(i) + 2
\end{equation*}
for all $i=0, \dotsc, g_y-2$, where $t(0)=2$. Hence
\begin{equation*}
r = t(g_y-1) \ge 2g_y.
\end{equation*}

Suppose, contrary to our claim, that $r \ge 2g_y + 1$. Consider the dual series
\begin{equation}\label{equation:g^s}
g_{g_x - 2g_y - 1}^s = \linsys{K_X - g_{g_x + 2g_y - 1}^r} = \linsys{K_X - D - f^{\ast}{K_Y}}.
\end{equation}
By the Riemann-Roch Theorem, we have $s = r - 2g_y \ge 1$. Since $\deg{g_{g_x - 2g_y - 1}^s} \le g_x - 2g_y$, it follows by the Castelnuovo-Severi inequality that $g_{g_x - 2g_y - 1}^s$ is induced by $f$, i.e,
\begin{equation}\label{equation:K_P-D-f^*K_Y=f^*E+P_i}
g_{g_x - 2g_y - 1}^s = f^{\ast}{\linsys{E}} + P_1 + \dotsb + P_t \end{equation}
for some effective $E \in \divisorgroup(Y)$, where $P_1, \dotsc, P_t \in X$ are (if any) base points.

By \eqref{equation:K_X-f^*K_Y+f^*B}, $K_X \lineqv f^{\ast}{K_Y} + f^{\ast}{B}$; it follows by \eqref{equation:g^s}, \eqref{equation:K_P-D-f^*K_Y=f^*E+P_i} that
\begin{equation*}
f^{\ast}{E} + P_1 + \dotsb + P_t \lineqv K_X - D - f^{\ast}{K_Y} \lineqv f^{\ast}{B} - D.
\end{equation*}
Therefore
\begin{equation*}
\linsys{D} = \linsys{f^{\ast}(B-E) - P_1 - \dotsb - P_t}.
\end{equation*}
On the other hand, by \eqref{equation:H^0(X, f^*(M))}, we have
\begin{equation*}
H^0(X, \sheaf{O_X}(f^{\ast}(B-E))) = H^0(Y,\sheaf{O_Y}(B-E) \oplus \sheaf{O_Y}(-E)) = H^0(Y, \sheaf{O_Y}(B-E)),
\end{equation*}
which implies that $\linsys{f^{\ast}(B - E)} = f^{\ast}{\linsys{B-E}}$; but then $\linsys{D}$ cannot be simple. This contradiction shows the claim.

Since $\dim{\linsys{D + (g_y-1)f^{\ast}{g_2^1(Y)}}} = 2 + 2(g_y-1)$ by Claim, it follows by Lemma~\ref{lemma:dim(g_l^t-g_m^1)} that $\dim\linsys{D - f^{\ast}{g_2^1(Y)}} = 0$. Therefore
\begin{equation*}
D = \linsys{f^{\ast}{g_2^1(Y)} + Q_1 + \dotsb + Q_{d'}}
\end{equation*}
for some $Q_i \in X$.
\end{proof}


\providecommand{\bysame}{\leavevmode\hbox to3em{\hrulefill}\thinspace}
\providecommand{\MR}{\relax\ifhmode\unskip\space\fi MR }
\providecommand{\MRhref}[2]{%
  \href{http://www.ams.org/mathscinet-getitem?mr=#1}{#2}
}
\providecommand{\href}[2]{#2}

\end{document}